\documentclass[a4paper,american,reqno]{amsart}

\usepackage{babel}
\usepackage[utf8]{inputenc}
\usepackage[T1]{fontenc}
\usepackage[binary-units=true]{siunitx}
\usepackage[ruled,linesnumbered]{algorithm2e}
\usepackage{csquotes}
\usepackage{xspace}
\usepackage{todonotes}
\usepackage{booktabs}
\usepackage[style = authoryear-comp,
            maxbibnames = 100,
            maxcitenames = 2,
            giveninits = true,
            uniquename = init,
            isbn = false,
            backend = bibtex]{biblatex}
\usepackage[colorlinks,
            citecolor=blue,
            urlcolor=blue,
            linkcolor=blue]{hyperref} 

\makeatletter
\patchcmd{\@settitle}{\uppercasenonmath\@title}{\scshape\large}{}{}
\patchcmd{\@setauthors}{\MakeUppercase}{\scshape\normalsize}{}{}
\makeatother

\vfuzz \hfuzz
\makeatletter
\@namedef{subjclassname@2020}{%
  \textup{2020} Mathematics Subject Classification}
\makeatother

\newtheorem{prop}{Proposition}

\SetKwInOut{Input}{Input}
\SetKwInOut{Output}{Output}

\newcommand{\st}{\text{s.t.}}
\newcommand{\define}{\mathrel{{\mathop:}{=}}}

\makeatletter
\newenvironment{varsubequations}[1]
{%
  \addtocounter{equation}{-1}%
  \begin{subequations}
    \renewcommand{\theparentequation}{#1}%
    \def\@currentlabel{#1}%

  }
  {%
  \end{subequations}\ignorespacesafterend
}
\makeatother

\newcommand\MyBox[2]{
  \fbox{\lower0.75cm
    \vbox to 1.7cm{\vfil
      \hbox to 1.7cm{\hfil\parbox{1.4cm}{#1\\#2}\hfil}
      \vfil}%
  }%
}

\newcommand{\defset}[3][\defsep]{\set{#2#1#3}}
\newcommand{\Defset}[3][\defsep]{\Set{#2#1#3}}
\newcommand{\set}[1]{\{#1\}}
\newcommand{\Set}[1]{\left\{#1\right\}}

\newcommand{\field}{\mathbb}
\newcommand{\naturals}{\field{N}}

\newcommand{\reals}{\field{R}}

\newcommand{\N}{\naturals}

\newcommand{\R}{\reals}

\newcommand{\AC}{\text{AC}\xspace}
\newcommand{\TP}{\text{TP}\xspace}
\newcommand{\TN}{\text{TN}\xspace}
\newcommand{\FP}{\text{FP}\xspace}
\newcommand{\FN}{\text{FN}\xspace}

\newcommand{\MCC}{\text{MCC}\xspace}
\newcommand{\CRF}{\textsf{C$^2$RF}\xspace}
\newcommand{\pCRF}{\textsf{p-C$^2$RF}\xspace}
\newcommand{\RF}{\textsf{RF}\xspace}
\newcommand{\codename}[1]{\textsf{#1}}

\makeatletter
\newcommand{\fcdot}{\,\cdot\,}
\newcommand{\fcarg}[1]{\def\fc@rg{#1}\ifx\fc@rg\empty\fcdot\else\fc@rg\fi}
\makeatother
\newcommand{\abs}[1]{\lvert\fcarg{#1}\rvert}

\newcommand{\card}{\abs}

\newcommand{\rev}[1]{#1}

\bibliography{semi-supervised-random-forests}

\begin{document}

\title[MILP for Cardinality-Constrained Random Forests]%
{Mixed-Integer Linear Optimization for Cardinality-Constrained Random
  Forests}

\author[J. P. Burgard, M. E. Pinheiro, M. Schmidt]%
{Jan Pablo Burgard, Maria Eduarda Pinheiro, Martin Schmidt}

\address[J. P. Burgard]{%
  Trier University,
  Department of Economic and Social Statistics,
  Universitätsring 15,
  54296 Trier,
  Germany}
\email{burgardj@uni-trier.de}

\address[M. E. Pinheiro, M. Schmidt]{%
  Trier University,
  Department of Mathematics,
  Universitätsring 15,
  54296 Trier,
  Germany}
\email{pinheiro@uni-trier.de}
\email{martin.schmidt@uni-trier.de}

\date{\today}

\begin{abstract}
  Random forests are among the most famous algorithms for solving
classification problems, in particular for large-scale data sets.
Considering a set of labeled points and several decision trees, the
method takes the majority vote to classify a new given point.
In some scenarios, however, labels are only accessible for a proper
subset of the given points.
Moreover, this subset can be non-representative, e.g., due to
collection bias. Semi-supervised learning considers the setting of
labeled and unlabeled data and often improves the reliability of the
results. In addition, it can be possible to obtain additional
information about class sizes from undisclosed sources. We propose a
mixed-integer linear optimization model for computing a
semi-supervised random forest that covers the setting of labeled and
unlabeled data points as well as the overall number of points in each
class for a binary classification. Since the solution time
rapidly grows as the number of variables increases, we present some
problem-tailored preprocessing techniques and an intuitive branching
rule.
Our numerical results show that our approach leads to better accuracy
and a better Matthews correlation coefficient for biased samples
compared to random forests by majority vote, even if only a few
labeled points are available.


\end{abstract}

\keywords{Random forests,
Semi-supervised learning,
Mixed-integer linear optimization,
Preprocessing,
Cardinality constraints%
%
%
}
\subjclass[2020]{90C11,
90C90,
90-08,
68T99
%
%
}

\maketitle

\section{Introduction}
\label{sec:introduction}

Random forests are one of the most famous approaches in supervised
learning \parencite{RFBreiman}.
It has been applied to various fields such as the prediction of
diseases \parencite{RFMed2,RFcorona1}, 3D object
recognition \parencite{RF3D} and Fraude and accident
detection \parencite{RFFraud,RFtraffic}.
The main reasons why random forests are popular are that they
prevent over-fitting \parencite{ElementsStatisticLearning},
that they have only a few parameters to tune, and that they can be used
directly for high-dimensional problems \parencite{RF1,RFguidetour}.
The core idea is, given labeled data, to combine the prediction of
different trees, in general, using the majority vote to classify new
points.

Nevertheless, acquiring labels for every unit of interest can be
costly---in particular when classic surveys are used to obtain the
labels.
In this situation, it would be beneficial to train the random forest
with only partly labeled data.
This yields a semi-supervised learning
setting \parencite{semisupervised}.
Algorithms for semi-supervised learning have already been proposed for
neural networks \parencite{SemiSupNN,NeuralSemiSUpervised,NeuralSEMISUP},
logistic regression \parencite{LRSemi,SemiSupRL},
support vector machines \parencite{Chapelle,LAPSVM2}, and decision
trees \parencite{Semisupervised2, SemiSupervised4, Semisupervised3}.

In the case of random forests, \textcite{SemiSupervisedRF} propose
an iterative and deterministic annealing-like training algorithm that
maximizes the multi-class margin of labeled and unlabeled
samples. Furthermore, \textcite{SemiSupervisedRFCO}  extend the
co-training paradigm to random forests, determining how certain the
model is about its predictions for unlabeled data.
Moreover, \textcite{SemiSupervisedRFA} combine active
learning and semi-supervised learning to improve
the final classification performance of random forests by utilizing
supervised clustering to categorize the unlabeled data.

However, in many applications, it is possible to know the total amount
of elements in each class within a population, e.g., when external
sources provide this information.
For instance, a company might only know the overall number of
successful transactions, but might not be able to identify which
specific customer's transactions were successful.
An intuitive example is an online retailer that may track the total
number of good customer reviews but does not have access to individual
ratings due to anonymity practices.
Another example is from healthcare, where it is possible to
know how many patients have a disease but, due to data privacy
reasons, one does not know which specific person is affected or not.
\textcite{Burgard2021CSDA} propose aggregating this extra information
for logistic regression.
They develop a cardinality-constrained multinomial logit model.
For support vector machines, \textcite{constrainedSVM} present a
mixed-integer quadratic optimization model and iterative clustering
techniques to tackle cardinality constraints for each class.
Moreover, for the case of decision trees,
\textcite{burgard2024mixedinteger} propose a mixed-integer linear
optimization model for computing semi-supervised optimal
classification trees that serve the same purpose.

Our contribution here is to propose a random forest model that imposes
a cardinality constraint on the classification of the unlabeled data.
We develop a big-$M$-based mixed-integer linear programming (MILP)
model to solve the cardinality-constrained random forest (C$^2$RF)
problem that includes the cardinality constraint for the unlabeled
data.
The cardinality constraint helps to account for biased samples since
the number of predictions in each class on the population is bounded
by the constraint.
In particular, our numerical results show that our approach
leads to better accuracy and a better Matthews correlation
coefficient for biased samples compared to random forests by majority
vote, even if only few labeled points are available.
The computation time for this MILP grows with the number of
variables---especially for an increasing number of integer
variables.
To account for this, we present theoretical results that lead
to preprocessing techniques that significantly reduce the computation
time.

This paper is organized as follows.
In Section~\ref{sec:MILP-formulation} we present the optimization
model and prove the correctness of the used big-$M$ parameter.
Afterward, the preprocessing techniques are discussed in
Section~\ref{sec:preprocessing} and an intuitive branching
rule is presented in Section~\ref{sec:priority}.
There, we also present our algorithm that combines the mentioned
techniques and the MILP formulation.
In Section~\ref{sec:numerical-results} we report and discuss numerical
results.
Finally, we conclude in Section~\ref{sec:conclusion}.


\section{An MILP Formulation for Cardinality-Constrained Random
  Forests}
\label{sec:MILP-formulation}

Let $X \in \R^{p \times N} = [X_u, X_l]$ be the data matrix with
unlabeled data $X_u = [x^1, \dotsc, x^{m}]$ and labeled data $X_l =
[x^{m+1}, \dotsc, x^N]$.
Hence, we are given points $x^i \in \R^p$ for all $i~\in~ [1,N]
\define \set{1,\dots, N}$.
We set $n \define N -m $ and $y \in \set{-1,1}^{n}$ as the vector of
class labels for the labeled data.
Let $t$ be the number of given decision trees and let $A^j \in
\R^{p\times d}$ be a subset of the labeled data with size $d$ for
$j\in [1,t]$.
For each $j \in [1,t]$, based on each column of $A^j$ and its
label, the $j$th tree generates a vector
$r^j \in \set{-1,1}^m$ that classifies the unlabeled data $X_u$.
Thus, for each unlabeled point $x^i$ we observe a vector of
classification $r_i = [r^1_i, \dots, r^t_i] \in \set{-1,1}^t$.
Hence, $R = [r_1, \dotsc, r_m] \in \set{-1,1}^{t\times m}$ and
$r_i^j$ is the classification of  $x^i \in X_u$ given by the tree~$j$.
In a random forest, the prediction for a point $x^i \in X_u$ is the
dominant class chosen by the individual $t$ trees, i.e., the majority
vote.

In many applications, aggregated information on the labels
is available, e.g., from census data.
For what follows, we assume to know the total number $\lambda \in \N$
of unlabeled points that belong to the positive class and propose a
model such that we can use a linear combination of the tree classifications
as well as $\lambda$ as an additional information.
Our goal is to find optimal parameters $\alpha^* \in \R^t$, $\eta^*
\in \R$, and $z^* \in \set{0,1}^m$ that solve the optimization problem
\begin{varsubequations}{P1}
  \label{eq:Problem1}
  \begin{align}
    \min_{\alpha,\eta,z} \quad
    & \eta \label{eq:RFfunction2}
    \\
    \st \quad
    &\alpha^\top r_i \leq -1 + z_i M, \quad i \in [1,m],
      \label{eq:linearcombination1} \\
    &  \alpha^\top r_i \geq 1 - (1-z_i)M, \quad i \in [1,m],
      \label{eq:linearcombination2} \\
    & \lambda - \eta \leq  \sum_{i=1}^m z_i \leq       \lambda + \eta,
      \label{unlabeled part2} \\
    &\ell \leq \alpha_j \leq u, \quad j \in [1, t] \label{eq:alpha}, \\
    & 0 \leq \eta \leq  \bar{\eta} , \label{eq:eta2} \\
    & z_i \in \set{0,1}, \quad i \in [1,m], \label{eq:binaryz}
  \end{align}
\end{varsubequations}
where $M$ needs to be chosen sufficiently large, $u > \ell >0$ holds,
and we set
\begin{equation}
  \label{bareta}
  \bar{\eta} \define \max \Set{\lambda, m - \lambda}.
\end{equation}

Note that the objective function in \eqref{eq:RFfunction2} minimizes
the classification error for the unlabeled data.
As $z_i$ is binary, Constraints~\eqref{eq:linearcombination1}
and~\eqref{eq:linearcombination2} lead to
\begin{align*}
  \alpha^\top r_i \geq 1 \implies z_i = 1, \quad i \in [1,m],\\
  \alpha^\top r_i \leq -1 \implies z_i = 0,\quad i \in [1,m].
\end{align*}

Constraint~\eqref{unlabeled part2} ensures that the number of
unlabeled data points classified as positive is as close to $\lambda$
as possible.
Constraint~\eqref{eq:alpha} bounds the weight of each tree's decision
for the final classification.
This means that for $j\in [1,t]$, as $\alpha_j$ gets closer to $u$, the
$j$th tree gets greater influence on the final classification, and as
$\alpha_j$ gets closer to $\ell$, the $j$th tree has less influence
on the final classification.
Observe that since  $\alpha_j \geq \ell > 0$ holds for all $j \in
[1,t]$, all trees contribute to the final classification.
Moreover, if $\alpha_j$ has the same value for all $j \in [1,t]$, all
trees contribute equally to the final classification and we are in the
standard random forest setup with majority vote.
Note that the upper bound $u$ is not necessary for the correctness of
the model but will serve as a
big-$M$-type parameter as can be seen in Proposition~\ref{BIGM} below.
The upper bound $\bar{\eta}$ in Constraint~\eqref{eq:eta2} is also
not necessary for the correctness of Model \eqref{eq:Problem1}.
Nevertheless, as can be seen in Proposition~\ref{propeta}, this upper
bound does not cut off any solution.
Hence, we include it in our implementation because we expect that the
solution process benefits from tight bounds. Problem~\eqref{eq:Problem1} is an MILP.
We refer to this problem as C$^2$RF
(Cardinality-Constrained Random Forest).
As usual for big-$M$ formulations, the choice of $M$ is crucial.
If $M$ is too small, the problem can become infeasible or optimal
solutions could be cut off.
If $M$ is chosen too large, the respective continuous relaxations
usually lead to bad lower bounds and solvers may encounter numerical
troubles.
The choice of~$M$ is discussed in the following proposition.

\begin{prop}
  \label{BIGM}
  A valid big-$M$ for Problem~\eqref{eq:Problem1} is given by $M =
  ut+1$, i.e., $M$ is linear in the number of trees in the forest.
\end{prop}
\begin{proof}
  For all $i \in [1,m]$ we have $r_i \in \set{-1,1}^t$.
  Moreover, Constraint~\eqref{eq:alpha} ensures that $\alpha_j \leq u$
  holds for all $j\in [1,t]$.
  Hence,
  \begin{equation*}
    \alpha^\top r_i  \leq \sum_{j=1}^t \alpha_j \leq ut
  \end{equation*}
  and
  \begin{equation*}
    \alpha^\top r_i \geq  -\sum_{j=1}^t \alpha_j \geq -ut
  \end{equation*}
  hold for all $i \in [1,m]$ and $ M = ut+1$ does not cut any
  feasible solution.
\end{proof}
The following proposition makes a statement about the upper bound
$\bar{\eta}$ in Constraint~\eqref{eq:eta2}.
\begin{prop}
  \label{propeta}
  Consider Problem~\eqref{eq:Problem1} in which
  Constraint~\eqref{eq:eta2} is replaced by $\eta \geq 0$.
  Then, for every $\eta^*$ as being part of an optimal solution, it
  holds $\eta^* \leq \bar{\eta}$ for~$\bar{\eta}$ as defined
  in~\eqref{bareta}.
\end{prop}
\begin{proof}
  Observe that since $z_i \in \set{0,1}$ for all $i \in [1,m]$,
  \begin{equation*}
    0 \leq \sum_{i=1}^m z_i \leq m
  \end{equation*}
  holds.
  Moreover, the maximum value occurs if all points are classified as
  positive.
  If this happens, from Constraint~\eqref{unlabeled part2} we obtain
  \begin{equation*}
    \eta \geq \sum_{i=1}^m z_i - \lambda = m -\lambda.
  \end{equation*}
  On the other hand, the minimum value of $ \sum_{i=1}^m z_i $ occurs
  if all points are classified as negative.
  If this happens, from Constraint~\eqref{unlabeled part2} we obtain
  \begin{equation*}
    \eta \geq \sum_{i=1}^m z_i + \lambda = \lambda.
  \end{equation*}
  Since Problem~\eqref{eq:Problem1} is a minimization Problem,
  $\eta \leq \bar{\eta}$ holds.
  Thus, the upper bound~$\bar{\eta}$ in Constraint~\eqref{eq:eta2}
  does not cut off any optimal point.
\end{proof}


\section{Preprocessing}
\label{sec:preprocessing}

In this section, we present different preprocessing techniques for
Problem~\eqref{eq:Problem1} that can be used to reduce the number of
binary as well as the number of continuous variables.

The first insight is that, if all trees have the same classification
for some unlabeled points, these points must have the same final
classification and, therefore, the respective binary variables
always have the same values.

\begin{prop}
  \label{prop:equalclass}
  Let $k \in [1,m]$ and consider $\mathcal{K} \define \defset{i \in
    [1,m]}{r_i = r_k}$.
  Then, $(\alpha,\eta, z)$ is a feasible point of
  Problem~\eqref{eq:Problem1} if and only if there exists a vector $\bar{z}
  \in \{0,1\}^{m+1-\vert \mathcal{K}\vert}$ such that $(\alpha, \eta,
  \bar{z})$ is a feasible point of the problem
  \begin{varsubequations}{P2}
    \label{eq:Problem1a}
    \begin{align}
      \min_{\alpha,\eta,z} \quad
      & \eta  \\
      \st \quad
      & \alpha^\top r_i \leq -1 + z_i M, \quad i \in \set{k} \cup [1,m]
        \setminus \mathcal{K}, \label{1a:linearcombination1}
      \\
     &   \alpha^\top r_i \geq 1 - (1-z_i)M,  \quad i \in \set{k} \cup
    [1,m] \setminus \mathcal{K}, \label{1a:linearcombination2}
      \\
      &   \lambda - \eta \leq  \sum_{i \in [1,m] \setminus \mathcal{K}} z_i
    + \vert \mathcal{K} \vert  z_k \leq \lambda + \eta,
    \label{unlabeled part3} \\
    & \eqref{eq:alpha},\eqref{eq:eta2}
      \\
   & z_i \in \set{0,1},
    \quad i \in \set{k} \cup [1,m] \setminus \mathcal{K}. \label{1a:z}
    \end{align}
  \end{varsubequations}
\end{prop}
\begin{proof}
  Consider $(\alpha,\eta,z)$ a feasible point of \eqref{eq:Problem1}
  and
  \begin{equation*}
    \bar{z}_i = z_i, \quad i \in  \set{k} \cup [1,m] \setminus \mathcal{K}.
  \end{equation*}
  We now prove that $(\alpha, \eta, \bar{z})$ is a feasible point of
  \eqref{eq:Problem1a}.
  Because \eqref{eq:linearcombination1}, \eqref{eq:linearcombination2},
  and \eqref{eq:binaryz} hold, \eqref{1a:linearcombination1},
  \eqref{1a:linearcombination2}, and \eqref{1a:z} are
  satisfied. Moreover, since for all $i \in \mathcal{K}$ it holds $r_i =
  r_k$, we obtain that $ \alpha^\top r_k = \alpha^\top r_i$ holds
  for all $i \in \mathcal{K}$. Hence, by
  Constraint~\eqref{eq:linearcombination1} and
  \eqref{eq:linearcombination2}, we obtain that $z_i = z_k$ also
  holds for all $i \in \mathcal{K}$.
  This together with $\bar{z}_k = z_k$ implies that
  $\card{\mathcal{K}} \bar{z}_k = \sum_{i \in \mathcal{K}} z_i$ is
  satisfied. Hence,
  \begin{equation}
    \label{eq:cardinality}
    \sum_{i \in [1,m] \setminus \mathcal{K}} \bar{z}_i
    + \vert \mathcal{K} \vert  \bar{z}_k  =  \sum_{i \in [1,m]
      \setminus \mathcal{K}} z_i + \sum_{i \in \mathcal{K}} z_i =
    \sum_{i=1}^m z_i
  \end{equation}
  is also satisfied, and, by Constraint~\eqref{unlabeled part2}, we
  obtain that Constraint~\eqref{unlabeled part3} holds.
  Therefore, $(\alpha, \eta, \bar{z})$ is a feasible point of
  Problem~\eqref{eq:Problem1a}.

  On the other hand, let $(\alpha, \eta, \bar{z})$ be a feasible
  point of Problem~\eqref{eq:Problem1a} and set
  \begin{equation}
    \label{eq:buildz}
    z_i = \begin{cases}
      \bar{z}_i, \text{ if } i \notin \mathcal{K},\\
      \bar{z}_k, \text{ otherwise}.
    \end{cases}
  \end{equation}
  Since \eqref{1a:linearcombination1}, \eqref{1a:linearcombination2}
  and \eqref{1a:z} are satisfied, \eqref{eq:linearcombination1} and
  \eqref{eq:linearcombination2} hold for each $i \notin \mathcal{K}$
  and \eqref{eq:binaryz} holds for all $i \in [1,m]$. Further, because
  each $i \in \mathcal{K}$ satisfies $r_i=r_k$,  $\alpha^\top r_i =
  \alpha^\top r_k$ holds for all $i \in \mathcal{K}$. Hence,  by
  Constraints~\eqref{1a:linearcombination1} and
  \eqref{1a:linearcombination2} we obtain that
  \begin{equation*}
    1-(1-z_i)M= 1-(1-\bar{z}_k)M \leq  \alpha^\top r_i \leq  -1+
    \bar{z}_k M = -1+z_i M
  \end{equation*}
  is satisfied for all $i \in \mathcal{K}$. Besides that, Expression
  \eqref{eq:buildz} implies that  $\vert \mathcal{K} \vert \bar{z}_k =
  \sum_{i \in \mathcal{K}} z_i$ and, therefore,
  Expression~\eqref{eq:cardinality} also holds. Hence,   by
  Constraint~\eqref{unlabeled part3}, we obtain that
  Constraint~\eqref{unlabeled part2} is satisfied.
  Therefore, $(\alpha, \eta,z)$ is a feasible point of
  Problem~\eqref{eq:Problem1}.
\end{proof}

The following proposition shows that if one or more trees classify all
points exactly as another tree, some continuous variables of
Problem~\eqref{eq:Problem1} can be eliminated.
\begin{prop}\label{prop:equaltrees}
  Given $g \in [1,t]$, consider $\mathcal{G} \define \defset{j \in
    [1,t]}{r^g = r^j}$.
  Then, $(\alpha^*, \eta^*, z^*)$ is a solution to
  Problem~\eqref{eq:Problem1} if and only if there exists a vector
  $\bar{\alpha} \in \mathbb{R}^{t+1-\vert \mathcal G \vert} $ such
  that $(\bar{\alpha}, \eta^*, z^*)$ is a solution to the problem
  \begin{varsubequations}{P3}
    \label{eq:Problem2}
    \begin{align}
      \min_{\alpha,\eta,z} \quad
      & \eta \
      \\
      \st \quad
      & \sum_{j \in [1,t] \setminus \mathcal{G}}\alpha_j r_i^j + \vert
        \mathcal{G} \vert \alpha_gr_i^g \leq -1 + z_i M,
        \quad i \in [1,m],
        \label{eq:linearcombination1p2}
      \\
      & \sum_{j \in [1,t] \setminus \mathcal{G}} \alpha_j r_i^j + \vert
        \mathcal{G} \vert \alpha_gr_i^g \geq 1 -(1- z_i) M,
        \quad i \in [1,m],
        \label{eq:linearcombination2p2}
      \\  & \eqref{unlabeled part2}, \\
      & \ell \leq \alpha_j \leq u,
        \quad j \in \set{g}\cup[1,t]\setminus\mathcal{G},
        \label{eq:alpha2}\\
         & \eqref{eq:eta2}, \eqref{eq:binaryz}.
    \end{align}
  \end{varsubequations}
\end{prop}
\begin{proof}
  Let $(\alpha^*, \eta^*, z^*)$ be a solution to
  Problem~\eqref{eq:Problem1} and
  \begin{equation*}
    \bar{\alpha}_ j=
    \begin{cases}
      \alpha_j^*,
      & \text{if } j \notin \mathcal{G},\\
      \sum_{j \in \mathcal{G}}\alpha_j^*/\vert  \mathcal{G} \vert ,
      & \text{otherwise.}
    \end{cases}
  \end{equation*}
  Since $\ell \leq  \alpha_j^* \leq u$ holds for all $j \in [1,t]$,
  we obtain
  \begin{equation*}
    \ell = \frac{\ell}{\vert\mathcal{G}\vert}( \vert\mathcal{G}\vert)
    \leq \bar{\alpha}_g
    \leq \frac{u}{\vert\mathcal{G}\vert}( \vert\mathcal{G}\vert)
    = u.
  \end{equation*}
  Moreover, because $r^g = r^j$ is satisfied for all $j \in
  \mathcal{G}$,
  \begin{equation}
    \label{prop31}
    \sum_{j \in \mathcal{G}}\alpha_j^*r_i^j
    = r_i^g \sum_{j \in \mathcal{G}}\alpha_j^*
    = \vert  \mathcal{G} \vert \bar{\alpha}_g r_i^g
  \end{equation}
  holds for all $i \in [1,m]$.
  Hence, for all $i \in [1,m]$,
  \begin{equation}
    \label{prop3p2}
    \sum_{j = [1,t] \setminus \mathcal{G}} \bar{\alpha}_j r_i^j
    + \vert \mathcal{G} \vert \bar{\alpha}_gr_i^g
    = \sum_{j = [1,t] \setminus \mathcal{G}} \alpha_j^* r_i^j
    + \sum_{j \in \mathcal{G}}\alpha_j^*r_i^j = (\alpha^*)^\top r_i
  \end{equation}
  is satisfied and, consequently,
  \begin{equation*}
    1 - (1-z_i^*)M
    \leq \sum_{j = [1,t] \setminus \mathcal{G}} \bar{\alpha}_j r_i^j
    + \vert \mathcal{G} \vert \bar{\alpha}_gr_i^g
    \leq -1 + z_i^* M
  \end{equation*}
  holds for $i \in [1,m]$.
  Therefore, $(\bar{\alpha}, \eta^*, z^*)$ is a solution to
  Problem~\eqref{eq:Problem2}.

  On the other hand, let $(\bar{\alpha}, \eta^*, z^*)$ be a solution
  to Problem~\eqref{eq:Problem2} and set
  \begin{equation*}
    \alpha^*_j =
    \begin{cases}
      \bar{\alpha}_j, & \text{if } j \notin \mathcal{G},\\
      \bar{\alpha}_g, & \text{otherwise}.
    \end{cases}
  \end{equation*}
  Since \eqref{eq:alpha2} holds, \eqref{eq:alpha} is satisfied for all
  $j \in [1,t]$.
  Besides that, since $r^g = r^j$ is satisfied for all $j \in
  \mathcal{G}$, \eqref{prop31} and \eqref{prop3p2} also hold for all
  $i \in [1,m]$.
  Hence, for all $i \in [1,m]$, we have
  \begin{equation*}
    1 - (1-z_i^*)M \leq  (\alpha^*)^\top r_i \leq -1 + z_i^* M
  \end{equation*}
  and $(\alpha^*, \eta^*, z^*)$ is a solution to
  Problem~\eqref{eq:Problem1}.
\end{proof}

Finally, the following proposition allows to fix some binary
variables~$z_i$ of Problem~\eqref{eq:Problem1}.

\begin{prop}
  \label{propfix}
  For each $i \in [1,m]$, consider
  $\mathcal{A}_i = \defset{j \in [1,t]}{r_i^j =-1}$
  and $\mathcal{B}_i = \set{j \in [1,t]: r_i^j =1}$.
  If for some $i \in [1,m]$,
  \begin{equation}
    \label{eq:bandc}
    \varphi_i \define -u \card{\mathcal{A}_i} + \ell
    \card{\mathcal{B}_i} \geq 1
  \end{equation}
  holds, then every feasible point $(\alpha, \eta, z)$ of
  Problem~\eqref{eq:Problem1} satisfies $z_i = 1$.
  If, on the other hand,
  \begin{equation}
    \label{eq:candb}
    \phi_i \define- \ell \card{\mathcal{A}_i} + u \card{\mathcal{B}_i} \leq -1
  \end{equation}
  is satisfied for some $i \in [1,m],$ then any feasible point
  $(\alpha, \eta, z)$ of Problem~\eqref{eq:Problem1} satisfies $z_i =
  0$.
\end{prop}
\begin{proof}
  Since $ \ell \leq \alpha_j \leq u$ is satisfied for all $j \in [1,t]$,
  if for some $i \in [1,m]$,
  \begin{equation*}
     -u\vert \mathcal{A}_i \vert + \ell \vert \mathcal{B}_i\vert \geq 1
  \end{equation*}
  holds, we obtain
  \begin{equation*}
    \alpha^\top r_i = -\sum_{j \in \mathcal{A}_i} \alpha_j + \sum_{j \in \mathcal{B}_i}
    \alpha_j
    \geq -u \vert \mathcal{A}_i \vert + \ell \vert \mathcal{B}_i \vert \geq 1,
  \end{equation*}
  and by Constraint~\eqref{eq:linearcombination1}, $z_i =1$.
  On the other hand, if for some $i \in [1,m]$, it holds
  \begin{equation*}
    - \ell \vert \mathcal{A}_i\vert + u \vert \mathcal{B}_i \vert  \leq -1,
  \end{equation*}
  we get
  \begin{equation*}
    \alpha^\top r_i = -\sum_{j \in \mathcal{A}_i} \alpha_j + \sum_{j \in \mathcal{B}_i}
    \alpha_j
    \leq - \ell \vert \mathcal{A}_i\vert + u \vert \mathcal{B}_i \vert  \leq -1,
  \end{equation*}
  and by Constraint~\eqref{eq:linearcombination2}, $z_i =0$.
\end{proof}

Consider now
\begin{equation*}
  \mathcal{P} \define \Defset{i\in [1,m]}{\varphi_i\geq 1}
  \quad \text{and} \quad
  \mathcal{N} \define \Defset{i\in [1,m]}{\phi_i \leq -1}.
\end{equation*}
Then, $\card{\mathcal{P}} + \card{\mathcal{N}}$ binary variables can
be fixed.
Moreover, $ \card{\mathcal{P}}$ points then are already classified as
positive.
If $ \vert{\mathcal{P}} \vert \geq \lambda$, due to cardinality
constraint, all remaining points $x^i \in X_u \setminus (\mathcal{P}
\cup \mathcal{N})$ must be classified as
negative, and $\lambda$ can be set to~$0$.
On the other hand, if $\card{\mathcal{P}} < \lambda$,
only $\lambda - \card{\mathcal{P}}$ points in
$X_u \setminus (\mathcal{P}
\cup \mathcal{N})$ must be classified
as positive.
This update is present in Step~\ref{algo:step18} in
Algorithm~\ref{algopre} below.


\section{Branching Priorities}
\label{sec:priority}

One aspect that can significantly affect the performance of MILP
solvers is the applied branching rule.
In this brief section, we present a problem-tailored rule for helping
the MILP code to solve Problem~\eqref{eq:Problem1}.
To this end, let us consider binary variables~$z_i, z_k \in \{0,1\}$,
$i,k \in [1,m]$, and positive integer values~$\xi_i$ and $\xi_k$ so that
$\xi_i > \xi_k$ implies that the solver should branch on $z_i$ before
$z_k$.
In our context, a point for which the percentage of trees that
classify the point as positive (or negative) is larger than for
another point seems to be ``easier'' to classify.
Hence, we want to branch on the respective binary variable first.
Based on that, we establish a criterion for a branching strategy.
We set $\theta_i = \card{\text{mean}( r_i )}$ for each $x^i \in X_u
\setminus (\mathcal{P} \cup \mathcal{N})$.
Observe that the higher the value of $\theta_i$, the more trees
classify the point~$x^i$ in one specific class.
Hence, we consider $\xi_i$ the position of $\theta_i$ in the vector of
the increasingly sorted values of $\theta$.
Thus, the higher the value of $\theta_i$, the higher the value of
$\xi_i$, and hence, the higher the branching priority for the binary
variable~$z_i$.

Motivated by the preprocessing techniques presented in
Section~\ref{sec:preprocessing} and the branching priority discussed
in this section, we obtain Algorithm~\ref{algopre} to
solve Problem~\eqref{eq:Problem1}.

\begin{algorithm}
  \label{algopre}
  \caption{p-C$^2$RF: Preprocessing and Solving C$^2$RF}
  \SetKwInput{Input}{Input~}
  \Input{$R\in \set{-1,1}^{t\times m}, u>\ell >0,$ $\lambda \in
    \mathbb{N}$, $\mathcal{K} = \emptyset$, $\beta =0$, and
    $\gamma=0$.}
  Compute $M = ut +1$ and  $\bar{R}=[\bar{r}_1,\dotsc, \bar{r}_h] \in
  \set{-1,1}^{t \times h}$ being the set of all different $r_i \in R.$
  \\
  \For{$k \in \set{1, \dotsc, h}$ }
  {Compute $w_k = \left\vert\{i \in [1,m]:r_i= \bar{r}_k
      \}\right\vert$, $ \varphi_k$ as described in \eqref{eq:bandc},
    and $\phi_k$ as described in \eqref{eq:candb}.
    \\
    \uIf{$\varphi_k \geq 1$}{Set $\mathcal{K}\leftarrow \mathcal{K}
      \cup \{k\}$ and $\beta \leftarrow  \beta+w_k$.}
    \ElseIf{$\phi_k \leq -1$}{Set $\mathcal{K}\leftarrow \mathcal{K}
      \cup \{k\}$ and $\gamma \leftarrow  \gamma+w_k$.}}
  Compute $S = [s^1, \dotsc, s^q]^\top \in \set{-1,1}^{q\times h}$
  being the set of all different $\bar{r}^j \in \bar{R}$.\\
  \For{$g \in \set{1, \dotsc, q}$}
  {Compute $v_g = \left\vert\{j \in
      [1,t]:r^j = \bar{r}^g \}\right\vert$ and
    set $s^g \leftarrow v_gs^g$.}
  \For{$i \in  \set{1, \dotsc, h }\backslash \mathcal{K}$}{Compute
    $ \theta_i = \vert \text{mean} (s_i) \vert $.}
  \For{$i \in  \set{1, \dotsc, h }\backslash \mathcal{K}$}
  {Compute $\xi_i$, i.e., the position of $\theta_i$ in the vector of the
    increasingly sorted values of $\theta$. \label{algo:stepBR}}
  Compute $\bar{\lambda} = \min\{0, \lambda -
  \beta\}$ and $\bar{\eta} = \max\{\bar{\lambda},m - \beta -\gamma-
  \bar{\lambda} \}$ and solve \label{algo:step18}
  \begin{align*}
    \min_{\alpha,\eta,z} \quad
    & \eta
    \\
    \st \quad
    &\alpha^\top s_i \leq -1 + z_iM,
      \quad i \in [1,h]\setminus\mathcal{K},
    \\
    & \alpha^\top s_i \geq 1 - (1-z_i)M,
      \quad i \in [1,h]\setminus\mathcal{K},
    \\
    & \bar{\lambda} - \eta \leq \sum_{i \in
      [1,h]\setminus\mathcal{K}} w_iz_i \leq \bar{\lambda} + \eta,
    \\
    & \ell \leq \alpha_j \leq u,
      \quad j \in [1, q],
    \\
    & 0 \leq \eta \leq  \bar{\eta} ,
    \\
    & z_i \in \set{0,1},
      \quad i \in [1,h]\setminus\mathcal{K}.
  \end{align*}
  with branching priorities $\xi$ to compute $\alpha^*,\eta^*,z^*$.
  \label{Step:Solve Problem}
\end{algorithm}



\section{Numerical Results}
\label{sec:numerical-results}

In this section, we present and discuss our computational results that
demonstrate the impact of considering the total amount of points in
each class and of using the preprocessing techniques as well as the
branching rule to speed up the solution process.

We illustrate this on different test sets from the literature.
The test sets are discussed in Section~\ref{subsec:test-sets}, while
the computational setup is described in
Section~\ref{subsec:computational-setup}.
The evaluation criteria are depicted in
Section~\ref{subsec:evaluation-criteria}.
Finally, the numerical results are discussed in
Section~\ref{sec:discussion_run_times} and
\ref{sec:discussion_accuracy_mcc}.

\subsection{Test Sets}
\label{subsec:test-sets}

For the computational analysis of the proposed approaches, we consider
the subset of instances presented by \textcite{Olson2017PMLB} that are
suitable for classification problems and that have at most three
classes and at least 5000 points.
Repeated instances are removed and instances with missing information
are reduced to the observations without missing information.
If three classes are given in an instance, we transform them into two
classes such that the class with label~1 represents the
positive class and the other two classes represent the negative
class. This results in a final test set of 13~instances, as listed in
Table~\ref{table1}. To avoid numerical instabilities, all data sets
are scaled as follows.
For each coordinate $j \in [1,p]$, we compute
\begin{equation*}
  l_j = \min_{i \in [1,N]}\Set{x^i_j},
  \quad
  u_j = \max_{i \in [1,N]}\Set{x^i_j},
  \quad
  m_j = 0.5 \left(l_j + u_j \right)
\end{equation*}
and shift each coordinate~$j$ of all data points~$x^i$ via $\bar{x}^i_j
= x^i_j - m_j$. Furthermore, if a coordinate~$j$ of the
re-scaled points is still large, i.e., if $\tilde{l}_j = l_j - m_j <-
10^{2}$ or $\tilde{u}_j = u_j - m_j > 10^{2}$ holds, it is re-scaled
via
\begin{equation*}
  \tilde{x}^i_j = (\bar{v} - \underline{v} ) \frac{\bar{x}^i_j -
    \tilde{l}_j}{\tilde{u}_j-\tilde{l}_j} + \bar{v}
\end{equation*}
with $\bar{v} = 10^2$ and $ \underline{v} = -10^{2}$.
The corresponding 4 instances that we re-scale are marked with an
asterisk in Table~\ref{table1}.
\begin{table}
  \caption{The entire test set with the number of
    points ($N$) and the dimension ($p$)}
  \label{table1}
  \begin{tabular}{c c c c }
    \toprule
    ID & Instance  &$ N $ &$p$ \\
    \midrule
    1         &  phoneme       &  5349  &  5  \\
    2                        &  magic                        &  18\,905 &  10 \\
    $3^*$ &  adult   &  48\,790 &  14   \\
    $4^*$ &  churn      &  5000  &  20   \\
    $5^*$ &  ring        &  7400  &  20   \\
    6 &  twonorm         &  7400  &  20   \\
    7 &  waveform\_21        &  5000  &  21   \\
    8 &  ann\_thyroid            &  7129  &  21   \\
    9 &  agaricus\_lepiota         &  8124  &  22   \\
    10 &  waveform\_40    &  5000  &  40   \\
    11 &  connect\_4                   &  67\,557 &  42   \\
    12 &  coil2000       &  8380  &  85   \\
    $13^*$ &  clean2                  &  6598  &  168  \\
    \bottomrule
  \end{tabular}
\end{table}


In our computational study, we focus on emphasizing the statistical
importance of cardinality constraints---mainly in the case of
non-representative biased samples.
Biased samples are highly recurrent in non-probability surveys, which
are surveys with an inclusion process that is not tracked and, hence, the
inclusion probabilities are unknown.
This means that correction methods such as inverse inclusion probability
weighting cannot be applicable. For a primer on inverse
inclusion probability weighting, we refer to \textcite{Skinner2011} and
the references therein.

To reproduce such a scenario, we create 5 biased samples with
\SI{1}{\percent} of the data being labeled for each instance.
Differently from a simple random sample, where each point has an equal
probability of being chosen as labeled data, in these biased samples the
labeled data is chosen with probability $\SI{85}{\percent}$ for
belonging to the positive class. Moreover, we consider
$t=20$ trees and for each $j \in [1,t]$, the size of the training subset
$A^j$ is $\SI{20}{\percent}$ of the labeled data. For each training
subset we use the \textsf{Decision Tree} package \parencite{DTJULIA}
to generate $r^j$.

In addition, in Appendix~\ref{sec:num-results-simple-sample}, we
provide the results under simple random sampling, which produces
unbiased samples. In this scenario, the results of the proposed
methods are similar to the random forest. Hence, there is no downside
to using the proposed method in case of an unknown sampling process.

\subsection{Computational Setup}
\label{subsec:computational-setup}

For each one of the 65 instances described in
Section~\ref{subsec:test-sets}, we compare the following approaches.
\begin{enumerate}
\item[(a)] \textsf{RF}: Random Forest by majority vote.
\item[(b)] \textsf{C$^2$RF} as given in Problem~\eqref{eq:Problem1} with
  $\bar{\eta}$ as defined in~\eqref{bareta} and $M$ from
  Proposition~\ref{BIGM}.
\item[(c)] \textsf{p-C$^2$RF} as described in Algorithm~\ref{algopre}.
\item [(d)]\textsf{only PP}: Algorithm~\ref{algopre} without the
  branching rule described in Step~\ref{algo:stepBR}.
\item[(e)] \textsf{only BR}:  \textsf{C$^2$RF} as given in
  Problem~\eqref{eq:Problem1} with the branching rule as described
  in Section~\ref{sec:priority} but without our problem-tailored
  preprocessing.
\end{enumerate}
Our comparison has been implemented in \codename{Julia}~1.8.5 and we
use \codename{Gurobi}~11.5 and
\codename{JuMP} \parencite{DunningHuchetteLubin2017} to solve
\textsf{C$^2$RF} as well as the MILP in Algorithm~\ref{algopre}. All
computations were executed on the high-performance cluster
``Elwetritsch'', which is part of the ``Alliance of High-Performance
Computing Rheinland-Pfalz'' (AHRP). We use a single Intel XEON SP
6126 core with \SI{2.6}{\giga\hertz} and \SI{64}{\giga\byte}~RAM as
well as a time limit of \SI{7200}{\second}.

Based on our preliminary experiments, for \textsf{C$^2$RF} and
\textsf{p-C$^2$RF} we set the bounds to $\ell = 1$ and $u = 100$.
Moreover, we set the \textsf{MIPFocus} parameter of \textsf{Gurobi}
to~3.

\subsection{Evaluation Criteria}
\label{subsec:evaluation-criteria}

The first evaluation criterion is the run time of the different
methods.
To compare run times, we use empirical cumulative distribution
functions (ECDFs).
Specifically, for $S$ being a set of solvers (or approaches as
above) and for $P$ being a set of problems, we denote by
$t_{p,s} \geq 0$ the run time of the approach~$s \in S$ applied to
the problem~$p \in P$ in seconds.
If $t_{p,s} > 7200$, we consider problem~$p$ as not
being solved by approach~$s$.
With these notations, the performance profile of approach~$s$ is
the graph of the function~$\gamma_{s} : [0, \infty) \to [0,1]$
given by
\begin{equation*}
  \gamma_{s}(\sigma) = \frac{1}{\vert P
    \vert}\card{ \Defset{p \in P}{t_{p,s} \leq \sigma} }.
\end{equation*}

Moreover, knowing the true label of all points, we categorize them
into four distinct categories: true positive (TP) or true negative
(TN) if the point is classified correctly in the positive or negative
class, respectively, as well as false positive (FP) if the point is
misclassified in the positive class and as false negative (FN) if the
point is misclassified in the negative class.
Motivated by this, we compute two classification metrics, for which a
higher value indicates a better classification.
The first one is accuracy ($\AC$).
It measures the proportion of correctly classified points and is given
by
\begin{equation}
  \label{AC}
  \AC \define \frac{\TP + \TN}{\TP + \TN + \FP + \FN} \in [0,1].
\end{equation}

The second metric is Matthews correlation coefficient ($\MCC$). It
measures the correlation coefficient between the observed and
predicted classifications and is computed by
\begin{equation}
  \label{MCC}
  \MCC \define \frac{\TP \times \TN - \FP \times \FN }{\sqrt{(\TP +
      \FP)( \TP + \FN)(\TN + \FP)(\TN + \FN)}} \in [-1,1].
\end{equation}
The main statistical question is the following: For a specific
instance, does using the cardinality constraint as additional
information increase the accuracy and the $\MCC$? Since
\textsf{C$^2$RF} \textsf{p-C$^2$RF}, \textsf{only PP} and \textsf{only BR} solve the same problem, we
only compare the difference of the accuracy and $\MCC$ according to
\begin{equation}
  \label{eq:comparing}
  \overline{\AC} \define \AC_{\pCRF}-\AC_{\RF},
  \quad
  \overline{\MCC} \define \MCC_{\pCRF}-\MCC_{{\RF}},
\end{equation}
where $\AC_{\RF}$ and $\AC_{\pCRF}$ are computed as in~\eqref{AC},
and $\MCC_{\RF}$ and $\MCC_{\pCRF}$ as in~\eqref{MCC}.

\subsection{Discussion of Run Times}
\label{sec:discussion_run_times}

Figure~\ref{fig:runtime} shows the ECDFs for the measured run
times. As expected, \textsf{RF} is the fastest algorithm because
it does not include any binary variable related to the unlabeled
points as the newly proposed models do. It can be seen that $\pCRF$
significantly outperforms $\CRF$. $\CRF$ solves only \SI{58}{\percent}
of the instances within the time limit, while $\pCRF$ solves
\SI{94}{\percent}. This shows that the preprocessing techniques and
the branching rule significantly decrease the run times.
However, by comparing the two lines for ``\textsf{only PP}'' and
``\textsf{only BR}'', we see that most of the speed-up is obtained by
the preprocessing techniques while the branching rule only helps to
improve the performance for a small amount of instances.
Since the branching rule is not harming and sometimes helps, we decide
to include it in what follows.
\begin{figure}
  \centering
  \includegraphics[width=0.6\textwidth]{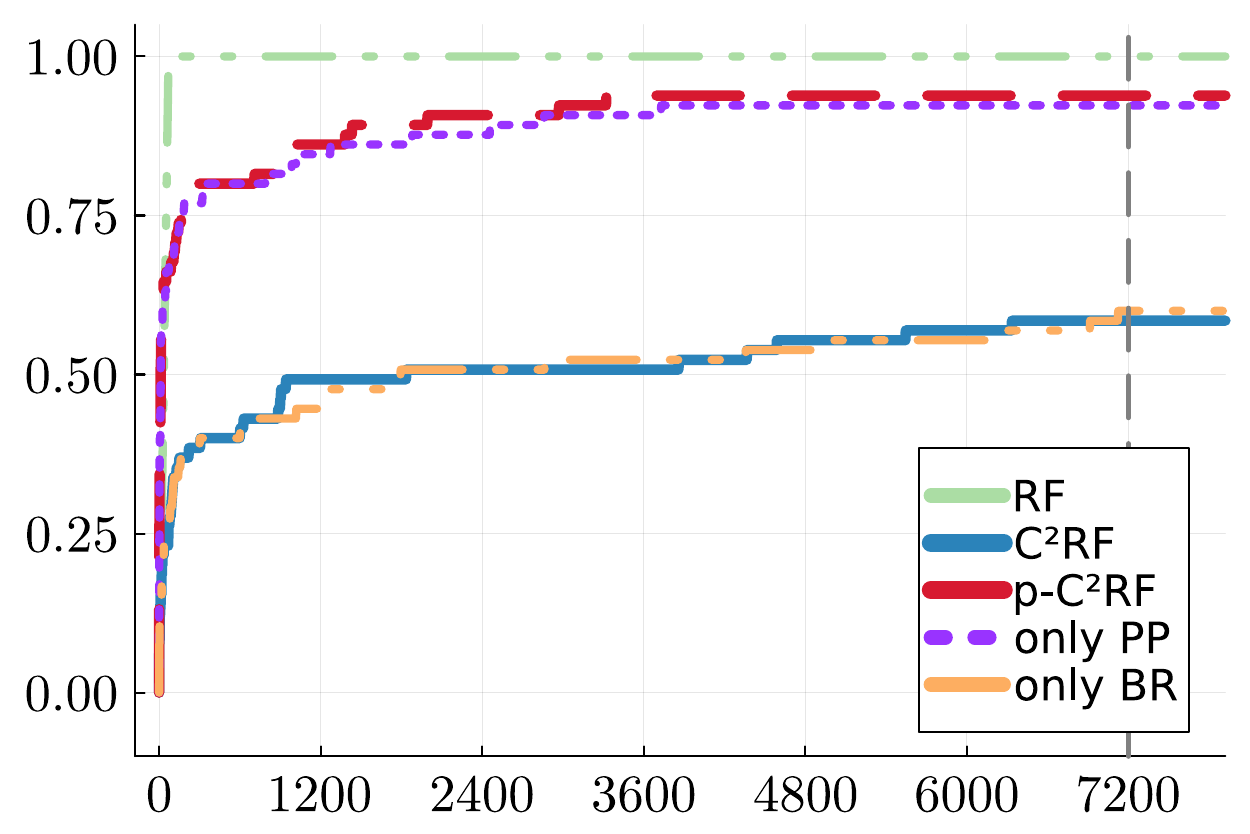}
  \caption{ECDFs for run times (in seconds)}
  \label{fig:runtime}
\end{figure}

\begin{table}
  \caption{Median of run times (in seconds)\rev{. The best value
    over } \rev{all semi-supervised approaches is printed in bold
    for every row of} \rev{ the table.}}
  \label{tab:mediantime}
  \begin{tabular}{l c c c c c}
    \toprule
    ID & \RF & \CRF & \pCRF & \textsf{only PP} &\textsf{only BR}\\
    \midrule
    1 & 0.042 &  3859.72 &  \bf{2.939} &  3.000 &   1249.15\\
    2 & 0.261 & --- & \bf{109.603} & 127.63  & ---  \\
    3 & 0.513 & --- &  \bf{2.865} & \bf{2.865} & ---  \\
    4 & 0.107 &  68.255  &  \bf{12.304} &  12.495 &  76.813 \\
    5 & 0.074 & --- &  \bf{999.798} & 1018.56 & ---  \\
    6 & 0.069 & --- &  --- & --- &  ---  \\
    7 & 0.172 & 1834.94 & \bf{4.054} &  4.149 &  ---  \\
    8 & 0.058 & 3.791 &  \bf{0.168} &  0.191 &  3.585\\
    9 & 0.087 &   899.891 & 1.599 &  \bf{1.568}  & 1018.80\\
    10 & 0.066  &  883.18 &  \bf{144.127} & 182.252 & --- \\
    11 & 1.129  &---  &   204.337 &  \bf{147.684} & ---   \\
    12 &  0.194  &   70.102 & 14.443 & \bf{12.619} & 75.937\\
    13 & 0.229  & 0.986 &  \bf{0.275} &  0.367 &  1.124 \\
    \bottomrule
  \end{tabular}
\end{table}

In Table~\ref{tab:mediantime} we present the median run times of the 5
biased samples for each instance. A ``---'' means that the approach
did not solve at least $3$ of the samples of the
instance within the time limit. \rev{One} can see that \RF almost always
takes less than \SI{1}{\second} to solve the problem. When comparing
the two novel approaches and only the instances that \CRF solves at
least one sample, Table~\ref{tab:mediantime} shows that our techniques
decrease the time computation by \SI{89}{\percent} on average.

\subsection{Discussion of Accuracy and MCC}
\label{sec:discussion_accuracy_mcc}

Observe that for both metrics $\overline{\AC}$ and $\overline{\MCC}$,
a value greater than zero indicates that $\pCRF$ had a better result
than RF. Besides that, the box in the boxplot depicts the range of the
medium \SI{50}{\percent} of the values; \SI{25}{\percent} of the
values are below and \SI{25}{\percent} are above the
box. Figure~\ref{fig:ac&mcc} shows that the $\overline{\AC}$ values
are greater than zero in \SI{75}{\percent} of the results (left
plot). Hence, our proposed method has better accuracy than the
conventional random forest. It can also be seen in
Figure~\ref{fig:ac&mcc} that the $\overline{\MCC}$ values are greater
than zero in most cases (right plot). Therefore, our method has a better
$\MCC$ than \RF.

\begin{figure}
  \centering
  \includegraphics[width=0.47\textwidth]{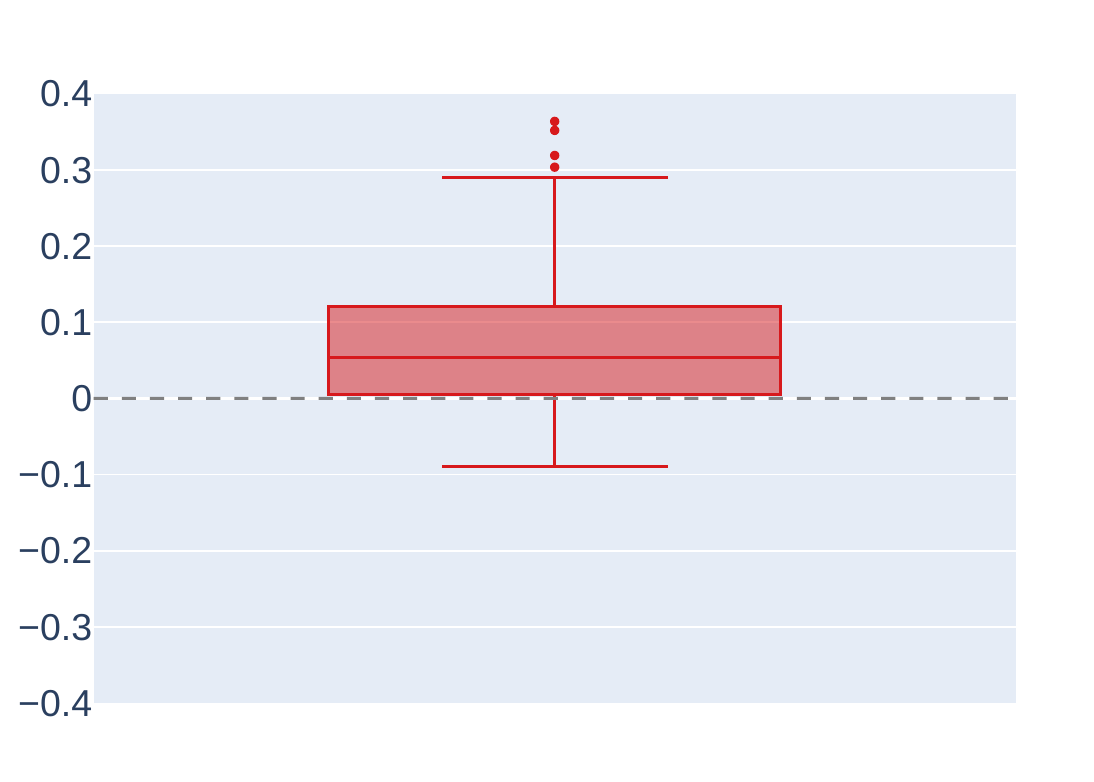}
  \includegraphics[width=0.47\textwidth]{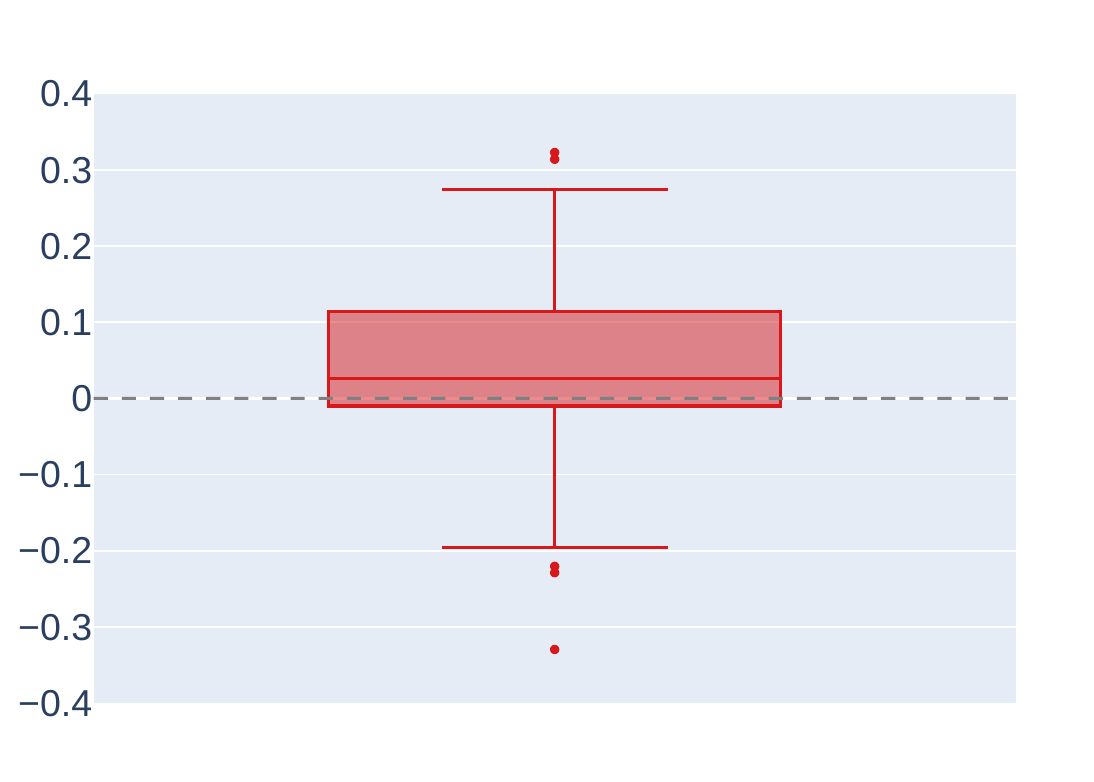}
  \caption{Comparison of $\overline{\AC}$ (left) and $\overline{\MCC}$
    (right); see~\eqref{eq:comparing}}
  \label{fig:ac&mcc}
\end{figure}

\begin{table}
  \caption{Median of \AC and \MCC (in percentage)\rev{. The best value}
    \rev{ for each metric is printed in bold for every row of the table.}}
  \label{tab:medianACandMCC}
  \begin{tabular}{l  c c    c c }
    \toprule
    ID &\multicolumn{2}{ c }{Accuracy} & \multicolumn{2}{ c }{\MCC} \\
    \midrule
       & \RF &  \pCRF & \RF & \pCRF\\
    \midrule
    1 & 62.16 & \bf{72.51} &  \bf{69.28}  & 69.20\\
    2 & 65.14 & \bf{75.03}&    70.35&   \bf{73.13}  \\
    3 & 76.28 &  \bf{77.78}   &  55.66 &    \bf{67.16}  \\
    4 &  61.98 &    \bf{79.64}&     55.17&   \bf{57.09}\\
    5 & 50.80 &   \bf{60.20} &   54.20&   \bf{61.45} \\
    6 &  58.90 & \bf{66.93}&   65.76&   \bf{67.05} \\
    7 & 75.88 &  \bf{78.59} &  \bf{78.55} &  76.47 \\
    8 &  98.58 &   \bf{98.74} &   84.03 &   \bf{84.73} \\
    9 &  81.30 &   \bf{87.59} &  83.95&   \bf{87.57} \\
    10 &  61.35 &  \bf{71.13} &  \bf{71.14}&   70.00 \\
    11 & 24.79 &   \bf{56.69} &  52.19&   \bf{55.98} \\
    12 & \bf{89.10} &    88.72&   \bf{54.61} &   53.57 \\
    13 & 99.43 & \bf{100} & 98.89 & \bf{100}\\
    \bottomrule
  \end{tabular}
\end{table}

When comparing each instance, Table~\ref{tab:medianACandMCC} shows the
median of \AC and \MCC of the 5 biased samples for RF and \pCRF. It
can be seen that, in the majority of instances, our approach has a
greater value of accuracy and \MCC than RF. Especially in terms of
accuracy, we obtained a better median value in 12 of the 13
instances. Regarding \MCC, our approach has a better median value than
\RF in 8 of the 13 instances. When \RF has better \MCC than \pCRF, it
is never better than \SI{2.5}{\percent}.

Figure~\ref{fig:ac&mcc} and Table~\ref{tab:medianACandMCC} show that
using the cardinality constraint of each class as additional
information allows to correctly classify the points with higher
accuracy and better $\MCC$ than with the RF by majority vote.


\section{Conclusion}
\label{sec:conclusion}

For several classification problems, it can be expensive to acquire
labels for the entire population of interest.
Nevertheless, external sources can, in some cases, offer additional
information on how many points are in each class.
For the case of binary classification, we proposed a semi-supervised
random forest that can be modeled using a big-$M$-based MILP
formulation.
We also presented problem-tailored preprocessing techniques and a
branching rule to reduce the computational cost of solving the MILP
model.

Under the condition of simple random sampling, our proposed
semi-supervised method has very similar accuracy and MCC as a
standard random forest.
In many applications, however, the available data come from
non-probability samples.
In this case, the data collection mechanism is largely unknown and
there is the risk of obtaining biased samples.
Our numerical results show that our model has better accuracy and MCC
than the conventional random forest even with a small number of
labeled points and biased samples.

\rev{Let us close with some comments on potential future work based on
  the results of this paper.
  First of all, the main ideas developed here can also be applied to
  other type of ensemble learning settings and is not restricted to
  the random forests considered here.
  Moreover, the resulting mixed-integer problems quickly turn out to
  be hard to solve.
  Hence, any further progress regarding additional tailored solution
  techniques like the ones presented in this paper will increase the
  practical applicability of the approach.}


\section*{Acknowledgements}

The authors thank the DFG for their support within RTG~2126
\enquote{Algorithmic Optimization}.


\section*{Availability of data and materials}
\label{sec:avail-data-mater}

The code developed for this paper is publicly available at
\url{https://github.com/mariaepinheiro/C2RF.jl} and all data can be
downloaded at \url{https://github.com/mariaepinheiro/C2RF}.


\printbibliography

\appendix
\section{Numerical Results for Simple random samples}
\label{sec:num-results-simple-sample}

In Section~\ref{sec:numerical-results} we present a computational
study on non-representative biased samples.
To complement our numerical results, we also present the results for
simple random sampling.
For simple random sampling, each element in the data set has the same
probability $(n/N)$ to be included in the sample of labeled data of
size~$n$.
The instances are the same as described in
Section~\ref{subsec:test-sets}. The computational setup follows the
description in Section~\ref{subsec:computational-setup}. As before,
the used evaluation criteria are $\overline{\AC}$ and
$\overline{\MCC}$ as in~\eqref{eq:comparing}.

\begin{figure}
  \centering
  \includegraphics[width=0.47\textwidth]{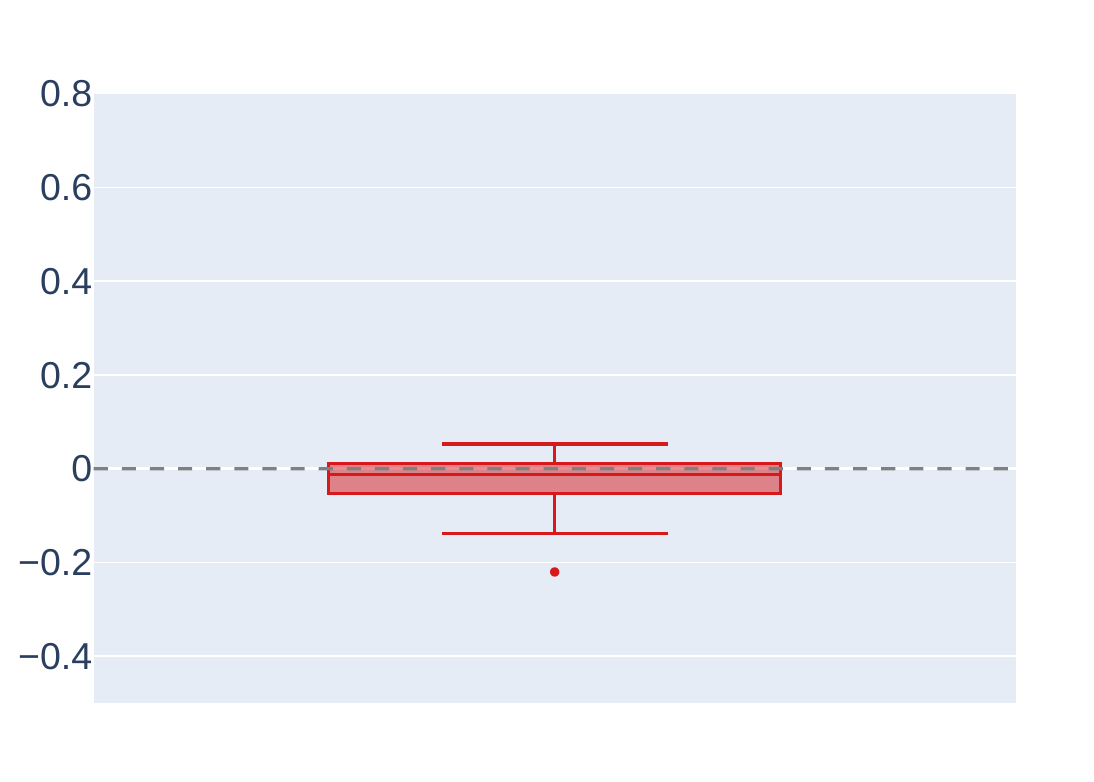}
  \includegraphics[width=0.47\textwidth]{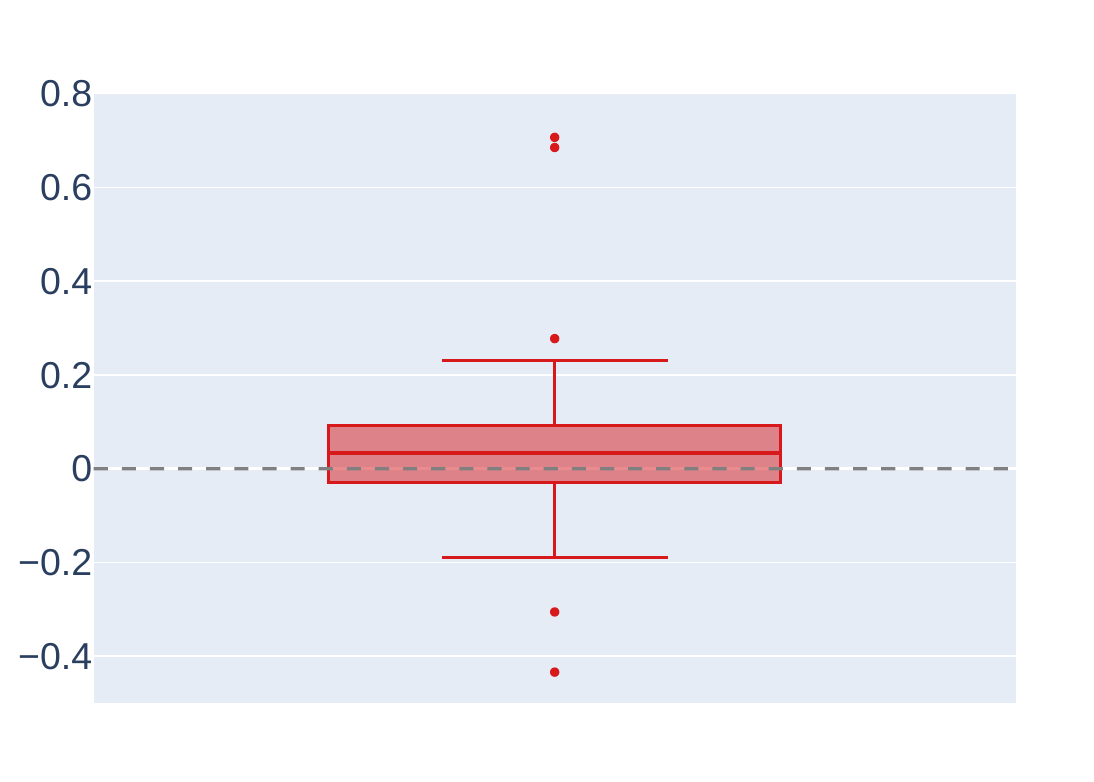}
  \caption{Comparison of $\overline{\AC}$ (left) and $\overline{\MCC}$
    (right); see~\eqref{eq:comparing}}
  \label{fig:ac&mcc-simplesample}
\end{figure}

It can be seen in Figure~\ref{fig:ac&mcc-simplesample} that
\SI{75}{\percent} of the values of $\overline{\AC}$ are between
$-0.05$ and $0.05$ (left plot). Figure~\ref{fig:ac&mcc-simplesample}
(right plot) also shows that $\overline{\MCC}$ has a value greater
than~0 and lower than~0 in \SI{50}{\percent} of the cases.

Table~\ref{tab:medianACandMCC-simplesample} shows the median of $\AC$
and $\MCC$ for each instance for $\pCRF$ and \RF.
In the majority of instances, our approach has a better or a very
similar accuracy and MCC compared to the conventional random forest.
Especially in terms of MCC, this is the case for all 13 instances.
From Figure~\ref{fig:ac&mcc-simplesample} and
Table~\ref{tab:medianACandMCC-simplesample} we can conclude that the
accuracy and MCC of our proposed method $\pCRF$ and the standard
random forest are very similar in the context of simple random
sampling.
This is expected because the cardinality constraint aims to
balance the class distribution and since the sample is not biased,
this constraint does not introduce additional meaningful information
to the problem.

\begin{table}
  \caption{Median of \AC and \MCC (in percentage)\rev{. The best value}
    \rev{ for each metric is printed in bold for every row of the table.}}
  \label{tab:medianACandMCC-simplesample}
  \begin{tabular}{l  c c    c c }
    \toprule
    ID &\multicolumn{2}{ c }{Accuracy} & \multicolumn{2}{ c }{\MCC} \\
    \midrule
       & \RF &  \pCRF & \RF & \pCRF\\
    \midrule
    1 & \bf{76.68}  &  76.32 & \bf{71.73} &  71.34 \\
    2 & \bf{78.28}  & 78.14 &  75.82&   \bf{75.86} \\
    3 & \bf{81.32} &   80.85 &  70.79&   \bf{73.70} \\
    4 & \bf{85.86} &   76.57&  50.0 &     \bf{51.63} \\
    5 & \bf{71.81}&    67.35 &    \bf{72.61}&   67.35 \\
    6 & 84.32 &   \bf{85.09} &   \bf{85.32} &   85.10 \\
    7 & 76.75 &  \bf{77.10}&  71.97&    \bf{74.10} \\
    8 & 97.68 &  \bf{98.61} &  50.0 &     \bf{84.43} \\
    9 & \bf{88.15} &   87.17 &  \bf{88.17}&   87.15 \\
    10 & 78.57 &  \bf{79.84} &  75.13 &   \bf{77.23} \\
    11 & \bf{75.36} &   67.71 &  50.0  &     \bf{55.89} \\
    12 & \bf{93.23} &  87.01&  50.0  &   \bf{52.62} \\
    13 & 97.64 &  \bf{100} &     95.37&  \bf{100}\\
    \bottomrule
  \end{tabular}
\end{table}


\end{document}